\documentclass[11pt]{article}
\usepackage[utf8]{inputenc}
\usepackage[T1]{fontenc}
\usepackage[english]{babel}
\usepackage{graphicx}
\usepackage{amssymb, amsmath, amsthm, mathrsfs, parskip, amsthm, enumerate, mathtools, subfig,comment}
\usepackage[toc,page]{appendix}
\usepackage[margin=0.3 in]{geometry}
\newgeometry{includefoot,left=3cm,right=3cm, bottom=3cm,top=3cm}
\usepackage{comment}

\newcommand{\R}{\mathbb{R}}

\newcommand{\dist}{\mathrm{dist}}

\providecommand{\ue}{u_{\epsilon}}
\providecommand{\Linf}{\mathcal{L}_{\infty}}
\providecommand{\Lpos}{\mathcal{L}_{\infty}^+}
\providecommand{\Lneg}{\mathcal{L}_{\infty}^-}
\providecommand{\Rn}{\mathbb{R}^n}

\newcommand{\inte}{\mathrm{int}}

\newcommand{\ol}{\overline}
\newcommand{\wt}{\widetilde}
\newtheorem{thm}{Theorem}[section]
\newtheorem{lem}[thm]{Lemma}

\newtheorem{cor}[thm]{Corollary}

\numberwithin{equation}{section}

\theoremstyle{definition}
\newtheorem{definition}{Definition}[section]

\title{On the comparison principle for a nonlocal infinity Laplacian}

\author{Frida Fejne}
\date{ }

\begin{document}

\maketitle
\thispagestyle{empty}

\begin{abstract}
   \noindent In this article, we prove the uniqueness of viscosity solutions to 
    $$
    \Linf u =f  \textup{ in }  \Omega, 
    $$
    where $\Linf$ denotes the nonlocal infinity Laplace operator, $\Omega$ a bounded domain, and $f$ a continuous functions such that $f \leq 0$. Uniqueness is established through a comparison principle.
\end{abstract}

\section{Introduction}
In this work, we consider uniqueness of viscosity solutions to Poisson's equation for the nonlocal infinity Laplacian where the right-hand side is a function that does not change sign. More precisely, we let $\Omega \subset \Rn$ be a bounded domain and consider the problem
\begin{equation}\label{eq:Luf}   
\Linf u =f  \textup{ in }  \Omega, 
\end{equation}
where $f \in C(\Omega)$ is a bounded function such that $f \leq 0$, and $\Linf u$ denotes the nonlocal infinity Laplace operator which is defined in \eqref{eq:infinitylap}.

Related problems have been studied extensively throughout the years. In 1967 in \cite{A1967} Aronsson introduced the fundamental operator known as the infinity Laplacian:
\begin{equation*}
    \Delta_{\infty} u(x) := D^2 u(x) Du(x) \cdot Du(x) =\sum_{i,j=1}^n D^2_{ij}u D_i u D_j u.
\end{equation*}
The equation $\Delta_{\infty} u(x)=0$ is the Euler-Lagrange equation when minimizing the supremum norm of the gradient of a function over a domain $\Omega$ with some boundary data $g$. The solution to this problem is called an infinity harmonic function and when $g$ is Lipschitz continuous, it will also be Lipschitz continuous. Furthermore, the infinity harmonic function with Lipschitz continuous boundary function $g$ will have the least Lipschitz norm among functions with the same boundary data. Therefore this minimization problem is known as the minimal Lipschitz extension problem. In his classic paper \cite{Jensen1993}, from 1993, Jensen proved that the minimizing functions are equivalent to viscosity solutions of the homogeneous infinity Laplace equation and established uniqueness. Since then, many researcher have studied and contributed to the theory of infinity harmonic functions. For instance, in \cite{BB2001} Barles and Busca presented a new and different proof of Jensen's theorem, which also worked for more general degenerate elliptic equations. For further references on the homogeneous infinity Laplace equation, see, e.g. \cite{AS2010, CEG2001, CGW2007, J1998, LM1995}, but we note that it is very far from a complete list.

Naturally, the infinity Laplace equation with inhomogeneous right-hand side has also been studied. In \cite{LW2008}, Lu and Wang worked with the Dirichlet problem
\begin{align*}  
\begin{cases}
\Delta_{\infty} u = f&  \textup{ in }  \Omega, \\
 u =g &  \textup{ on }  \partial \Omega.
\end{cases}
\end{align*}
 They established the existence and uniqueness of viscosity solutions under the assumption that the right-hand side $f$ is continuous and strictly positive or strictly negative, that is, it stays away from zero. In addition, they also provided a counterexample of the uniqueness when $f$ changes its sign. To our knowledge, the case where $f \geq 0$, is still an open problem. The work of Lu and Wang was continued in \cite{JPR2016} by Juutinen, Parviainen and Rossi. They worked with the gradient constraint problem
\begin{align*}  
\begin{cases}
 \textup{min}\{\Delta_{\infty} u(x), |Du(x)|-\chi_D \}=0&  \textup{ in }  \Omega, \\
 u =g&  \textup{ on }  \partial \Omega.
\end{cases}
\end{align*}
Here, $\chi_D$ denotes the characteristic function of a set $D \subset \Omega$.
Under the relatively mild regularity assumption that $\ol{D}=\ol{\inte \, D}$ they proved the uniqueness of viscosity solutions to the boundary value problem above. For more references on the inhomogeneous infinity Laplace equation, see, e.g. \cite{BDM1991,CP2007,IL2005,Y2010}.

 One nonlocal version of a fractional $p$-Laplace problem in a bounded domain $\Omega$ can be defined as
\begin{align} \label{eq:fracomega}
     \begin{cases}
                (-\Delta_p)^{s}u(x)  = 0 & \textup{ in } \Omega, \\
                    u =g & \textup{ on } \partial \Omega, 
     \end{cases}
\end{align}
where $s \in (0,1)$ and
\begin{equation}\label{eq:fraclap}
(-\Delta_p)^{s}u(x) = \lim_{\epsilon \to 0} \int_{\Omega\setminus B_{\epsilon(x)}} \frac{|u(x)-u(y)|^{p-2}(u(x)-u(y))}{|x-y|^{n+s p}} \ dy.
\end{equation}
Here $n$ denotes the dimension. We call this operator a fractional $p$-Laplace operator. We note that there is no standard definition when discussing this operator and the definitions in the literature varies, see e.g.,  \cite{STD2018} for a different version. As $p$ tends to infinity, the limit problem in \eqref{eq:fracomega} becomes
\begin{align} \label{eq:holderinf}
     \begin{cases}
               (L_{\infty})^{\alpha} u(x)=0  & \textup{ in } \Omega, \\
                    u =g & \textup{ on } \partial \Omega, 
     \end{cases}
\end{align}
 where we have used the notation $\alpha$ instead of $s$, for $\alpha \in (0,1]$ and
$$
(L_{\infty})^{\alpha} u(x) := \sup_{y \in \ol{\Omega}} \frac{u(y)-u(x)}{|x-y|^{\alpha}} + \inf_{y \in \ol{\Omega}} \frac{u(y)-u(x)}{|x-y|^{\alpha}}.
$$
Note that this equation is meaningful for $\alpha=1$. This operator is called the Hölder infinity Laplacian. The reason behind this name is that when the boundary function $g$ is Hölder continuous, more specifically $g \in C^{0, \alpha}(\partial \Omega)$, the solution $u_{\infty}$, to \eqref{eq:holderinf}, has the smallest Hölder seminorm among functions with the same boundary data. Thus, the solution is the optimal Hölder extension. This problem and the corresponding non-homogeneous Dirichlet problem were first introduced in \cite{CLM2012} where Chambolle et al. proved the existence of and some regularity results for solutions to $(L_{\infty})^{\alpha} u = f$. Note that this equation is, in general, not obtained as a limit of the $p$-Laplacian unless $f=0$. In the case where $f=0$ they derived a representation formula for the solutions and proved the uniqueness through a comparison principle. 

In \cite{LL2012} Lindgren and Lindqvist worked on an eigenvalue problem involving the operator in \eqref{eq:fraclap} and functions with compact support. They also studied the passage to infinity, and after letting $p \to \infty$ they obtained an expression involving
\begin{equation} \label{eq:infinitylap}
(\mathcal{L}_{\infty})^{\alpha}u(x) = \sup_{y \in \Rn} \frac{u(y)-u(x)}{|x-y|^{\alpha}} + \inf_{y \in \Rn} \frac{u(y)-u(x)}{|x-y|^{\alpha}},
\end{equation}
for some $\alpha \in (0,1]$. This operator is called a nonlocal infinity Laplacian. It is also worth noting that the nonlocal infinity Laplacian can be obtained as a limit of more general
fractional operators, see \cite{BPLS2022} for details. Furthermore, we note that the definition of the nonlocal infinity Laplacian operator may vary; see, e.g., \cite{BCF2012} for a different version.
In \cite{FL2016} the authors extended the problem in \eqref{eq:fracomega} to include functions with noncompact support and a more general right-hand side. They studied weak and viscosity solutions to
\begin{align*} 
     \begin{cases}
                    - (\mathcal{L}_p)^{s} u(x) = f(x,u) & \textup{ in } \Omega, \\
                    u =g & \textup{ on } \Rn \setminus \Omega, 
     \end{cases}
\end{align*}
where
$$
(\mathcal{L}_p)^{s} u(x) = \lim_{\epsilon \to 0} \int_{\Rn \setminus B_{\epsilon}(x)} 2\frac{|u(y)-u(x)|^{p-2}(u(y)-u(x))}{|x-y|^{n+s p}} \ dy,
$$
 for different functions $f$ on the right-hand side. Here $0<s<1$ and $\Omega$ is a Lipschitizian bounded domain. They also studied the limit equation after letting $p$ tend to infinity and obtained an expression involving the nonlocal infinity Laplacian. For more related work on the fractional $p$-Laplacian and the nonlocal infinity Laplacian, see e.g., \cite{ CRY2010, CS2007,SR2019}.  

The operator in \eqref{eq:infinitylap} is fundamental for our problem. We note that it depends on $\alpha$ but since it will be clear from context what $\alpha$ is, we will just denote it by $\mathcal{L}_{\infty}u(x)$. The nonlocal infinity Laplacian can be decomposed as 
$$
\mathcal{L}_{\infty}u(x) = \Lneg u(x) + \Lpos u(x),
$$
where
\begin{align*}
  \mathcal{L}^-_{\infty}u(x):=  \inf_{y \in \Rn} \frac{u(y)-u(x)}{|x-y|^{\alpha}}
   \quad \textup{and} \quad \mathcal{L}^+_{\infty}u(x):=  \sup_{y \in \Rn} \frac{u(y)-u(x)}{|x-y|^{\alpha}}.
\end{align*}
Since the operator is not on divergence form there is no natural weak formulation of the problem in \eqref{eq:Luf} and we therefore consider viscosity solutions, see Definition \ref{def:viscositySol}. In this work, we show a comparison principle for viscosity solutions to \eqref{eq:Luf} in the case where $0 < \alpha \leq 1$. This is Theorem \ref{thm:comparison}. We are not considering existence, and therefore we do not work directly with the boundary function $g$. However, we impose some conditions on our functions outside the domain $\Omega$, namely that they are bounded on $\Rn \setminus \Omega$ and that they have limits at infinity. Furthermore, we also assume that the right-hand side $f\in C(\Omega)$ is bounded and does not switch sign, that is, $f \leq 0$. The same results can also be obtained using similar methods if $f$ is nonnegative. The proof of Theorem \ref{thm:comparison} uses ideas from \cite{CLM2012} and a key factor is that
the operator $ \mathcal{L}^-_{\infty}u(x)$ attains its infimum on $\Rn \setminus \Omega$ for viscosity supersolutions to \eqref{eq:Luf}. This is shown in Lemma \ref{lem:Linfcont} which is based on Lemma 26 in \cite{LL2012}. The main difference from the comparison theorem in \cite{CLM2012} is that we are working in $\Rn$ and we use the infimal convolution of $u$ instead of the mollification of $u$. Furthermore, our right-hand side is not necessarily zero so $\mathcal{L}^+_{\infty}u(x)$ is not necessarily attained on $\Rn \setminus \Omega$. 

\subsection*{Acknowledgments}
We thank Erik Lindgren for suggesting the problem, as well as for his help and encouragement. The research was supported by VR Grant 2016-03639. We extend our gratitude to Mikko Parviainen for suggesting that the main result might hold in a slightly more general setting. We also thank two anonymous referees for their feedback on the paper.

\section{Properties of viscosity supersolutions}
\label{sec:supersol}
In this section, we are going to analyze viscosity supersolutions to 
\begin{equation}\label{eq:Luf2}   
\Linf u=f  \textup{ in }  \Omega.
\end{equation}
As stated in the introduction, we assume that the right-hand side $f$ is a bounded continuous function such that $f(x) \leq 0$. However, we note that for the proofs of all the lemmas in this section we actually only need $f$ to be nonpositive and in Lemma \ref{lem:infconv} $f$ can be any function. We also note that this definition only focuses on the equation in $\Omega$ and does not involve an explicit boundary function. When we prove the comparison principle in Section \ref{sec:comparison} we assume that $u$ is bounded on $\Rn \setminus \Omega$ and that $u$ has a limit at infinity, but we make no further assumptions on the boundary data. We start by giving the definition of a viscosity supersolution to \eqref{eq:Luf2}.

\begin{definition} \label{def:viscositySol}
Let $u : \R^n \to \R$ be a lower semicontinuous (resp. upper semicontinuous) function such that $|u| \leq C(1+|x|)^{\beta}$ for some $C> 0$ and $\beta < \alpha$. Furthermore, let $x_0 \in \Omega$ and let $\varphi(x)$ be a locally Lipschitz continuous function, such that $|\varphi| \leq C(1+|x|)^{\beta}$ for some $C> 0$ and $\beta < \alpha$, and such that $\varphi$ touches $u$ from below (resp. above) at $x_0$. That is,  $\varphi(x_0)=u(x_0)$ and 
     $$
     u \geq \varphi \quad \textup{in } \Rn \quad(\textup{resp. } u \leq \varphi).
     $$
   We say that $u$ is a viscosity supersolution (resp. subsolution) to \eqref{eq:Luf2} in $\Omega$ if 
    $$
    (\Linf \varphi)(x_0) \leq f(x_0) \qquad (\textup{resp. }(\Linf \varphi)(x_0) \geq f(x_0)),
    $$
where $f \in C(\Omega)$ is a bounded function such that $f \leq 0$. 

Finally, we say that a continuous function $u$ is a viscosity solution to \eqref{eq:Luf2} if it is both a viscosity supersolution and a viscosity subsolution.
\end{definition}

 In the following lemma we show that if a supersolution attains its minimum in the domain $\Omega$, then the function is constant, i.e., a version of the strong maximum principle.
\begin{lem}
\label{lem:strongmaxprinc}
   Let $u$ be a viscosity supersolution to $\Linf u = f$ in $\Omega$. Furthermore, assume that there exists some $x_0 \in \Omega$ such that $u(y) \geq u(x_0)$ for all $y \in \Rn \setminus \Omega$. Then $u$ is a constant function.
\end{lem}

\begin{proof}
    We begin by assuming that $u$ is not constant in $\Rn$ because otherwise there is nothing to show. Next, we subtract the constant $u(x_0)$ from $u$ and denote the new function by $\tilde u(x): = u(x)-u(x_0)$. We note that
\begin{align} \label{eq:strongmax}
     \begin{cases}
                   \tilde u(x)\geq 0 & \textup{on} \ \Rn \setminus \Omega, \\
                    \Linf \tilde u \leq 0 & \textup{in} \ \Omega, \\
                    \tilde u(x_0) = 0.
     \end{cases}
\end{align}
The aim is to prove that \eqref{eq:strongmax} implies that $\tilde u \equiv 0$ in $\Rn$ and we start by showing that $\tilde u\geq 0 $ in $\Omega$. We argue by contradiction and therefore assume that
$\inf_{x \in \Omega} \tilde u < 0$. Note that since $\tilde{u} \geq 0$ on $\Rn \setminus \Omega$ the infimum must be attained for some $\tilde x \in \Omega$. Because if not, there exists some sequence $x_k$ such that $x_k \to \tilde{x} \in \partial \Omega$ with $\lim_{k \to \infty} \tilde u(x_k) < 0$. But since $\tilde u$ is lower semicontinuous we know that $\liminf_{k \to \infty} u(x_k) \geq u(\tilde{x})$ and $u(\tilde{x}) \geq 0$ according to \eqref{eq:strongmax}. This is a contradiction and therefore there is a $\tilde{x} \in \Omega$ such that $\tilde{u}(\tilde{x})<0$. Next, since $\tilde u$ is a supersolution we know that $\Linf \psi(\tilde{x}) \leq 0$ for any test function $\psi$ that touches $\tilde u$ from below at $\tilde x$ i.e., 
\begin{equation}
\label{eq:help}
\Linf \psi(\tilde{x}) = \inf_{y \in \Rn} \frac{\psi(y)-\psi(\tilde{x})}{|\tilde{x}-y|^{\alpha}} + \sup_{y \in \Rn} \frac{\psi(y)-\psi(\tilde{x})}{|\tilde{x}-y|^{\alpha}}  \leq 0.
\end{equation}
We recall that $\tilde{x}$ is a minimum for $\tilde u$ and $\tilde{u}(x) \geq 0$ on $\Rn \setminus \Omega$. We can thus choose a Lipschitz continuous test function 
\begin{align*}
   \psi(x) =  \begin{cases}
                   \tilde u(\tilde{x})\ & \textup{in} \ \ol{\Omega}, \\
                    \tilde u(\tilde{x}) < \psi(x) \leq 0 & \textup{on} \ \Rn \setminus \ol{\Omega},
     \end{cases}
\end{align*}
and note that for $y \in \Rn  \setminus \ol{\Omega}$ we have that $\psi(y)-\psi(\tilde x)>0$. But then the infimum in \eqref{eq:help} is greater than or equal to zero while the supremum must be strictly greater than zero. However, this means that $\Linf \psi(\tilde{x}) >  0$ which is a contradiction. Thus, $\tilde u \geq 0$ in $\Omega$ and therefore $\tilde{u}\geq0$ in $\Rn$.

Since we have assumed that $u$ and therefore $\tilde{u}$ is not constant and $\tilde{u} \geq 0$ there exists a $\hat{x} \in \Rn$ such that $\tilde{u}(\hat{x}) >0.$ If we choose a nonnegative test function $\varphi$ that touches $\tilde u$ from below at $x_0$ we see that
\begin{equation*}
    \Linf \varphi(x_0) = \inf_{y \in \Rn} \frac{\varphi(y)}{|y-x_0|^{\alpha}} + \sup_{y \in \Rn} \frac{\varphi(y)}{|y-x_0|^{\alpha}} \geq \sup_{y \in \Rn} \frac{\varphi(y)}{|y-x_0|^{\alpha}},
\end{equation*}
since $\varphi(x_0)=\tilde{u}(x_0)=0$ and $\varphi(x) \geq 0$. But since we assumed that $\tilde{u}(\hat{x})>0$ it follows that we can choose a test function $\varphi$ which is not identically equal to zero but positive in at least some points. This follows since $\tilde{u}$ is lower semicontinuous so the set 
$$
\tilde{u}^-((\tilde{u}(\hat{x})/2, \infty)) = \{x \in \Rn:\tilde{u}(x)> \tilde{u}(\hat{x})/2 \},
$$
is open and this set is clearly not empty since $\tilde{u}(\hat{x})/2 >0$. Thus, $\Linf \varphi(x_0)>0$ which is a contradiction since $\tilde u$ is a supersolution. We therefore conclude that $\tilde u \equiv 0$ in $\Rn$ and it follows that $u(x)=u(x_0)$, i.e., $u$ is a constant function.
\end{proof}
Next, we take $x \in \Omega$ and define the operator
$$
L^-u(x) = \inf_{y \in \Rn \setminus \Omega} \frac{u(y)-u(x)}{|y-x|^{\alpha}}.
$$
For viscosity supersolutions to \eqref{eq:Luf2} it turns out that $\Lneg u(x) =L^- u(x)$. This is the content of the next lemma and it will be crucial when we prove the uniqueness in Theorem \ref{thm:comparison}. The ideas in the proof are from Lemma 26 in \cite{LL2012}.
\begin{lem}
\label{lem:Linfcont}
Let $u$ be a bounded viscosity supersolution to $\Linf u = f$ in $\Omega$. Then, for all $x \in \Omega$, we have that
\begin{equation}
\label{eq:Linf}
    \Lneg u(x) = \inf_{y \in \Rn \setminus \Omega} \frac{u(y)-u(x)}{|y-x|^{\alpha}},
\end{equation}
i.e., the infimum is attained in the complement of $\Omega$.
\end{lem}

\begin{proof}
We take $x_0 \in \Omega$ and split the proof into the two cases $L^- u(x_0) < 0$ and $L^- u(x_0) \geq 0$.

\textbf{Case 1:} $L^- u(x_0) < 0$. \\
We define
$$
w(x) = u(x_0) + L^- u(x_0) C_{x_0,R}^{\alpha}(x),
$$
where 
\begin{align*}  
C_{x_0,R}^{\alpha}(x) =
\begin{cases}
\textup{min}\{|x_0-x|^{\alpha}, R^{\alpha} \} &  \textup{ for }  0 < \alpha <1 \textup{ and }R>0, \\
\textup{min}\{|x_0-x|-\epsilon |x-x_0|^2, R-\epsilon R^2\}  &  \textup{ for }  \alpha = 1 \textup{ and } \epsilon R<1.
\end{cases}
\end{align*}
 If $\Lneg u(x_0)$ is attained for $x \in \Rn \setminus \Omega$ there is nothing to show and we therefore assume that $\Lneg u(x_0)$ is attained for $x \in \Omega$. We choose $R$ so large that $\Omega \subset \subset B_R(x_0)$.
If we can prove that $w(x) \leq u(x)$ for $x \in \Omega$, we are done, since then it follows that
$$
\frac{u(x)-u(x_0)}{C_{x_0,R}^{\alpha}(x)} \geq L^-u(x_0).
$$
 We note that
\begin{enumerate}
\item $u(x_0) = w(x_0)$,
\item From the definition of $w(x)$ it follows that
\begin{align*}
    u(x)-w(x) & = C_{x_0,R}^{\alpha}(x) \left(\frac{u(x)-u(x_0)}{C_{x_0,R}^{\alpha}(x)}- \inf_{y \in \Rn \setminus \Omega} \frac{u(y)-u(x_0)}{|y-x_0|^{\alpha}} \right) \geq 0,
\end{align*}
 on $ \Rn \setminus \Omega$.
\item From Lemma 24 and Lemma 25 in \cite{LL2012} we know that $\Linf  C_{x_0,R}^{\alpha}(x) < 0$ in $\Omega \setminus \{x_0\}$ and since
 $$
 \mathcal{L}_{\infty}w(x) = L^-u(x_0)\mathcal{L}_{\infty} C_{x_0,R}^{\alpha}(x),
 $$
 and $L^- u(x_0)< 0$ by assumption, it follows that 
 $$\Linf w(x) > 0 \textup{ for } x \in \Omega \setminus \{x_0\}.
 $$
 \item $\Linf u(x) \leq f(x) \leq 0$ for $x \in \Omega$, in the viscosity sense since $u$ is a viscosity supersolution.
\end{enumerate}

These four points imply that $w(x) \leq u(x)$. Because if not, there exists at least some point $\hat{x}$ in $\Omega \setminus \{x_0\}$ such that $u(\hat{x}) < w(\hat{x})$. We define
$$
d = \sup_{x \in \Omega} \left(w(x)-u(x)\right),
$$
and note that $d >0$. We let $\tilde{x}$ be the point where the supremum is achieved. We note that $\tilde{x}$ exists due to point 2 and the fact that $u$ is lower semicontinuous. We then choose the test function $\varphi(x) = w(x)- (w(\tilde{x})-u(\tilde{x}))$ and note that due to point 1 and 2, and the way that we chose $\tilde{x}$, it follows that $\varphi$ touches $u$ from below at $\tilde{x}$. Furthermore, because of point 1 we know that $\tilde{x} \neq x_0$. According to point 4 we must have that $\Linf w(\tilde{x}) \leq 0$ but this contradicts the third point i.e., that $\Linf w(\tilde{x}) > 0$ in $\Omega \setminus \{x_0\}$. Therefore, $w(x) \leq u(x)$ which implies \eqref{eq:Linf}, and thus the lemma is proved.

\textbf{Case 2:} $L^- u(x_0) \geq 0$. \\
In this case, $u(y) \geq u(x_0)$ for all $y \in \Rn \setminus \Omega$ and it follows from Lemma \ref{lem:strongmaxprinc} that $u(x_0)=C$ in $\Rn$. In this case $\Lneg u(x_0) = L^-u(x_0)=0$ so the result holds trivially. 
\end{proof}

We are now going to define the infimal convolution of a bounded and lower semicontinuous function $u$. 

\begin{definition}
\label{def:infconv}
    Let $u$ be bounded and lower semicontinuous. For $\epsilon > 0$ and $x \in \Rn$, we define the infimal convolution of $u$ as 
\begin{equation}\label{eq:infconv}
      u_{\epsilon}(x) \coloneqq \inf_{y \in \Rn} \left(\frac{|x-y|^2}{2 \epsilon} + u(y)\right) .    
\end{equation}
\end{definition}
We will call the point in which the infimum in \eqref{eq:infconv} is attained for $x^*$. Clearly the sequence $u_{\epsilon}(x)$ increases pointwise towards $u(x)$ as $\epsilon \to 0$. It is well known that the function
$$
v_{\epsilon}(x) = u_{\epsilon}(x)-\frac{|x|^2}{2 \epsilon},
$$
is concave and therefore $u_{\epsilon}$ is locally Lipschitz continuous and has second Alexandrov derivatives. For more information about the infimal convolution, see e.g., \cite{JP2012}. We proceed by showing that the infimal convolution of a viscosity supersolution to \eqref{eq:Luf2} is a supersolution in the classical sense. This proof is similar to the proof of Proposition 1 in \cite{LM2006}.

\begin{lem}
\label{lem:infconv}
Let $u(x)$ be a viscosity supersolution to \eqref{eq:Luf2} such that $0 \leq u(x) \leq L$, for $x \in \Omega$. We take $\epsilon >0$ and define $\Omega_{\epsilon}$ as 
$$
\Omega_{\epsilon} = \{x \in \Omega: \dist(x, \partial \Omega)> r(\epsilon)\},
$$
where $r(\epsilon) \coloneqq \sqrt{2L\epsilon}$. Then, $u_{\epsilon}$ is a supersolution of
    \begin{equation*}
        \Linf \ue (x) = \tilde{f}^{\epsilon}(x),
    \end{equation*}
in $\Omega_{\epsilon}$ in the classical sense where
$$
\tilde{f}^{\epsilon}(x) =  f(x^*).
$$
\end{lem}

\begin{proof}
 Let $x_0 \in \Omega$ be such that $\dist(x_0, \partial \Omega)> \sqrt{2L\epsilon}$. Furthermore, let $x_0^*$ be the point where the infimum is attained, i.e., 
   \begin{align}\label{eq:test1}
         u_{\epsilon}(x_0) = \frac{|x_0-x_0^*|^2}{2 \epsilon} + u(x_0^*).
    \end{align}
    We note that, since $u \geq 0$ in $\Omega$,
\begin{equation*}
    \frac{|x_0-x_0^*|^2}{2\epsilon} \leq \frac{|x_0-x_0^*|^2}{2\epsilon} +u(x_0^*) = u_{\epsilon}(x_0) \leq u(x_0) \leq L.
\end{equation*}
Thus, we see that $|x_0-x_0^*| \leq r(\epsilon)$, that is, $x_0^* \in B_{r(\epsilon)}(x_0)$. By definition we have
    \begin{align}\label{eq:test2}
        u_{\epsilon}(x) \leq \frac{|x-y|^2}{2 \epsilon} + u(y),
    \end{align}
    for all $y \in \Rn$. We set 
    \begin{equation*}
        \psi(x) = u_{\epsilon}(x-z)-\frac{|z|^2}{2 \epsilon},
    \end{equation*}
    for $z= x_0^*-x_0$. Next we note that
    \begin{align*}
        \psi(x_0^*) = u_{\epsilon}(x_0) - \frac{|x_0-x_0^*|^2}{2 \epsilon} = u(x_0^*),
    \end{align*}
    where we used \eqref{eq:test1} for the last equality. We also note that according to \eqref{eq:test2} it follows that
    \begin{align*}
        u_{\epsilon}(x-z) \leq \frac{|x-z-y|^2}{2 \epsilon} + u(y),
    \end{align*}
    for any $y \in \Rn$. If we choose $y=x$ we see that $u(x) \geq u_{\epsilon}(x-z)-\tfrac{|z|^2}{2 \epsilon} = \psi(x)$. Thus, $\psi(x)$ touches $u(x)$ from below at $x_0^*$ and since $u_{\epsilon}$ is locally Lipschitz continuous it can be used as a test function for $u$ at $x_0^*$. Since $u_{\epsilon}(y) = \psi(y+z) + \tfrac{|z|^2}{2 \epsilon}$ we see that
    \begin{align*}
        \Lneg u_{\epsilon}(x_0) &= \inf_{y \in \Rn} \frac{u_{\epsilon}(y) -u_{\epsilon}(x_0)}{|x_0-y|^{\alpha}} = \inf_{y \in \Rn} \frac{\psi(z+y)-\psi(x_0+z)}{|x_0-y|^{\alpha}} \\ 
        &=  \inf_{y \in \Rn} \frac{\psi(z+y)-\psi(x_0^*)}{|x_0^*-(z+y)|^{\alpha}} 
        = \inf_{\xi \in \Rn} \frac{\psi(\xi)-\psi(x_0^*)}{|x_0^*-\xi|^{\alpha}} \\
    &= \Lneg \psi(x_0^*).
    \end{align*}
    In the same way, it follows that $\Lpos u_{\epsilon}(x_0) = \Lpos \psi(x_0^*)$ and therefore $\Linf u_{\epsilon}(x_0)= \Linf \psi(x_0^*)$. Thus,
    $$
    \Linf u_{\epsilon}(x_0) \leq f(x_0^*) =: \tilde{f}^{\epsilon}(x_0),
    $$
    which proves the lemma.
\end{proof}

\section{Uniqueness}\label{sec:comparison}
We are now ready to show the uniqueness of viscosity solutions to \eqref{eq:Luf}.
\begin{thm}\label{thm:comparison}
    Let $u$ and $v$ be a viscosity supersolution and a viscosity subsolution to \eqref{eq:Luf2}, respectively, where $0 < \alpha  \leq 1$, and $f \leq 0$ is bounded and continuous in $\Omega$. Furthermore, let $u(x)$ be bounded on $ \Rn \setminus \Omega$ and $u \geq v$ on $\Rn \setminus \Omega$. In addition, we assume that $u(x)$ and $v(x)$ have limits, as $x$ tends to infinity, in the sense that $\liminf_{|x| \to \infty} u(x) \geq C_1$ and $\limsup_{|x| \to \infty} v(x) \leq C_2$, where $C_1 \geq C_2$. Then $u \geq v$ in $\Omega$.
\end{thm}

The structure and main ideas of the proof are as follows: We start by assuming that $u$ is bounded in $\Omega$ and that the claim of the theorem is not true. Since $u$ is a supersolution, we can use Lemma \ref{lem:Linfcont} to conclude that for $x \in \Omega$, $\Lneg u(x)$ is attained outside of $\Omega$. This makes it possible for us to modify $u$ so that it becomes a strict supersolution, which together with the fact that $v$ is a subsolution allows us to reach a contradiction. However, the function $u$ is not smooth enough and, therefore, we need to use the infimal convolution $u_{\epsilon}$ throughout the proof. The ideas remain the same, but as a result the proof will be more technically complicated. By the end of the proof we use that the right-hand side $f$ is bounded in $\ol{\Omega}$ in order to remove the assumption that $u$ is bounded.

\begin{proof}[Proof of Theorem \ref{thm:comparison}] We divide the proof into three steps.

\textbf{Step 1. Preliminaries} \\
At first we assume that $u$ is bounded in $\Omega$ i.e., $|u(x)| \leq \wt{L}$ for some constant $\wt{L} > 0$ whenever $x \in \Omega$. By adding a constant (to both $u$ and $v$) we may assume that $0\leq u(x)\leq L$ for some constant $L > 0$. We define
\begin{align*}
    M = \sup_{x \in \Rn} (v(x)-u(x)),
\end{align*}
and set
\begin{align*}
    K_0 = \{x \in \Rn:v(x)-u(x) = M \}.
\end{align*}
For the sake of contradiction, we assume that
\begin{equation}\label{eq:contradiction2}
    M \geq \gamma > 0.
\end{equation}
It follows from Lemma \ref{lem:A1} that $M$ must be finite and $K_0 \subset \subset \Omega$.
For some small enough $\beta>0$ we consider a compact $\beta$-neighbourhood $K(\beta)$ of $K_0$ such that
$$
K_0 \subset \subset K(\beta) \subset \subset \Omega,
$$
and for $x \in \Rn$ we define the infimal convolution of $u$ as in Definition \ref{def:infconv}. We recall that as $\epsilon \to 0^+$ it follows that $u_{\epsilon}(x) \to u(x)$ from below and by Lemma~\ref{lem:infconv}, $u_{\epsilon}$ is a supersolution to $\Linf u_{\epsilon} =  \tilde{f}^{\epsilon}$ in 
$$
\Omega_{\epsilon} = \{x \in \Omega:\dist(x, \partial \Omega) > r(\epsilon) \},
$$
where $r(\epsilon) = \sqrt{2L\epsilon}$, and 
$$
 \tilde{f}^{\epsilon}(x) =  f(x^*) \leq \sup_{y \in B_{r(\epsilon)}(x)} f(y).
$$
Here $x^*$ is the point where the infimum is attained for the infimal convolution of $u$ at the point $x$. We choose $\epsilon$ small enough so that $K_0 \subset \subset K(\beta) \subset \subset \Omega_{\epsilon}$ and proceed to define
\begin{align*}
    M_{\epsilon} &= \sup_{x \in \Rn}(v(x)-u_{\epsilon}(x)) , \\
    K_{\epsilon} &= \{x \in \Rn  :v(x)-u_{\epsilon}(x) = M_{\epsilon} \}.
\end{align*}
It follows from Lemma \ref{lem:A2} that $K_{\epsilon} \subset \subset K(\beta)$ for sufficiently small values of $\epsilon$.

We next consider $x_{\epsilon} \in K_{\epsilon}$. Since $\underbrace{v(x_{\epsilon})-u_{\epsilon}(x_{\epsilon})}_{M_{\epsilon}} \geq v(x)-u_{\epsilon}(x)$ for all $x \in \Rn$ it follows that
$$
v(x) \leq u_{\epsilon}(x)+M_{\epsilon},
$$
for all $x \in \Rn$.
We set
\begin{align*}
 \varphi_{\epsilon} =  u_{\epsilon}+M_{\epsilon},
\end{align*}
and thus
\begin{align}
     \begin{cases}
                 v \leq  \varphi_{\epsilon}  \ \textup{ in  } \Rn, \\
                 v = \varphi_{\epsilon} \ \textup{ in  }  K_{\epsilon}.
     \end{cases}
\end{align}
According to Lemma \ref{lem:A3} there exist an $\epsilon_0>0$ and a $\tau_0 >0$ such that for $\epsilon < \epsilon_0$, and $\tau \leq \tau_0$ we have that
\begin{equation}
\label{eq:blahblah}
v \leq u_{\epsilon} + M_{\epsilon}/8 \ \textup{ on } \ \Rn \setminus \Omega_{\tau},   
\end{equation}
where
$$
\Omega_{\tau}:= \{ x \in \Omega, d(x, \partial \Omega)> \tau \}. 
$$

We choose $\tau$ so small that $K(\beta) \subset \subset \Omega_{\tau}$ and take $\epsilon$ such that $r(\epsilon) < \tau$ so that $\partial \Omega_{\epsilon} \subset V$ where $V = \Rn \setminus \Omega_{\tau}$. Let $\eta_{\epsilon}(x) \in C^{\infty}(\Rn)$ be such that
\begin{align*}
     \begin{cases}
                   \eta_{\epsilon} = 1 & \textup{on  } \ \Rn \setminus \Omega_{\epsilon}, \\
                    0 \leq \eta_{\epsilon} \leq 1 & \textup{in  } \Omega_{\epsilon} \cap V, \\
                    \eta_{\epsilon} =0  & \textup{in  } \ \Omega_{\epsilon} \setminus V = \Omega_{\tau},
     \end{cases}
\end{align*}
and note that
$$
\Omega_{\epsilon} \cap V = \{x \in \Omega: r(\epsilon) < d(x, \partial \Omega) < \tau \}.
$$
Next, we define
\begin{equation*}
     \tilde{\varphi}_{\epsilon} = 
                \varphi_{\epsilon} -\frac{M_{\epsilon}}{4} \eta_{\epsilon},
\end{equation*}
and note that $\tilde{\varphi}_{\epsilon}= \varphi_{\epsilon}$ on $\Omega_{\epsilon} \setminus V$. Furthermore, we recall that 
$$
K_{\epsilon} \subset \subset K(\beta) \subset \subset \Omega_{\tau} = \Omega_{\epsilon} \setminus V,
$$
so $\tilde{\varphi}_{\epsilon} = \varphi_{\epsilon}$ in $K_{\epsilon}$. From equation \eqref{eq:blahblah} we know that
$$
v(x) \leq u_{\epsilon}(x) + M_{\epsilon}/8=\varphi_{\epsilon}(x)-\frac{7 M_{\epsilon}}{8} < \tilde{\varphi}_{\epsilon}(x),
$$
when $x \in V$. Therefore
\begin{align*}
     \begin{cases}
                   v \leq \tilde{\varphi}_{\epsilon} &\textup{in  } \Rn, \\
                   v= \tilde{\varphi}_{\epsilon} &\textup{in  } \ K_{\epsilon},
     \end{cases}
\end{align*}
and we conclude that $\tilde{\varphi}_{\epsilon}$ is a test function for $v$ at $x_{\epsilon} \in K_{\epsilon}$, i.e.,
\begin{equation}
\label{eq:contradiction}
\Linf \tilde{\varphi}_{\epsilon}(x_{\epsilon}) \geq f(x_{\epsilon}).    
\end{equation}
\textbf{Step 2. Obtaining a strict supersolution}\\
Next, we note that
\begin{align}
    \Lpos \tilde{\varphi}_{\epsilon}(x_{\epsilon}) &= \sup_{y \in \Rn} \frac{\tilde{\varphi}_{\epsilon}(y)-\tilde{\varphi}_{\epsilon}(x_{\epsilon})}{|y-x_{\epsilon}|^{\alpha}} = \sup_{y \in \Rn} \frac{\varphi_{\epsilon}(y)-\varphi_{\epsilon}(x_{\epsilon})-\eta_{\epsilon} (y)M_{\epsilon}/4}{|y-x_{\epsilon}|^{\alpha}} \nonumber \\
    & \leq \Lpos \varphi_{\epsilon}(x_{\epsilon}). \label{eq:Lplus}
\end{align}
 Furthermore, 
\begin{align}
    \Lneg \tilde{\varphi}_{\epsilon}(x_{\epsilon}) &= \Lneg \left(\varphi_{\epsilon}-\frac{M_{\epsilon}}{4} \eta_{\epsilon} \right)(x_{\epsilon}) \nonumber \\  
    &= \inf_{y \in \Rn} \frac{\varphi_{\epsilon}(y)-\frac{M_{\epsilon}}{4}\eta_{\epsilon} (y)-\varphi_{\epsilon}(x_{\epsilon})+\frac{M_{\epsilon}}{4}\eta_{\epsilon} (x_{\epsilon})}{|y-x_{\epsilon}|^{\alpha}}\nonumber   \\
    &=\inf_{y \in \Rn} \frac{\varphi_{\epsilon}(y)-\varphi_{\epsilon}(x_{\epsilon})-\frac{M_{\epsilon}}{4}\eta_{\epsilon} (y)}{|y-x_{\epsilon}|^{\alpha}} \nonumber \\
    &\leq  \inf_{y \in \Rn \setminus \Omega_{\epsilon}} \frac{\varphi_{\epsilon}(y)-\varphi_{\epsilon}(x_{\epsilon})-\frac{M_{\epsilon}}{4}}{|y-x_{\epsilon}|^{\alpha}}. \nonumber  
\end{align}
We recall that $u_{\epsilon}(x)$ is a supersolution so by definition $\varphi_{\epsilon}(x)$ is a supersolution. If $\Lneg \varphi_{\epsilon}(x_{\epsilon})$ is attained for some $y_{\epsilon} \in \Rn \setminus \Omega_{\epsilon}$ it follows that
\begin{equation}\label{eq:wantthis}
   \Lneg \tilde{\varphi}_{\epsilon}(x_{\epsilon}) \leq \inf_{y \in \Rn \setminus \Omega_{\epsilon}} \frac{\varphi_{\epsilon}(y)-\varphi_{\epsilon}(x_{\epsilon})-\frac{M_{\epsilon}}{4}}{|y-x_{\epsilon}|^{\alpha}} \leq  \Lneg \varphi_{\epsilon}(x_{\epsilon})-\delta_{\epsilon},
\end{equation}
where $\delta_{\epsilon} = \frac{M_{\epsilon}}{4|y_{\epsilon}-x_{\epsilon}|^{\alpha}}$. Note that in the second inequality we used that
$$
\inf_{y \in \Rn} \frac{\varphi_{\epsilon}(y)-\varphi_{\epsilon}(x_{\epsilon})}{|y-x_{\epsilon}|^{\alpha}} = \inf_{y \in \Rn \setminus \Omega_{\epsilon}} \frac{\varphi_{\epsilon}(y)-\varphi_{\epsilon}(x_{\epsilon})}{|y-x_{\epsilon}|^{\alpha}},
$$
which is due to Lemma \ref{lem:Linfcont} since $u_{\epsilon}$ is a supersolution in $\Omega_{\epsilon}$. Next, we show that there exists some $\tilde{\delta}>0$ such that
\begin{equation}
\label{eq:delta_bound}
    \delta_{\epsilon} = \frac{M_{\epsilon}}{4|y_{\epsilon}-x_{\epsilon}|^{\alpha}} > \tilde{\delta},
\end{equation}
for $x_{\epsilon} \in K_{\epsilon}$ and very small values of $\epsilon$. If this is not true we either have that $|y_{\epsilon}|=\infty$ for all $\epsilon$ or there exists a sequence $x_{\epsilon}$ such that the corresponding sequence $y_{\epsilon} \to \infty$ as $\epsilon \to 0$. We start with the case $|y_{\epsilon}|=\infty$ for all $\epsilon$ less than some $\epsilon_0$. If this is true, then
$$
 \inf_{y \in \Rn \setminus \Omega_{\epsilon}} \frac{u_{\epsilon}(y)-u_{\epsilon}(x_{\epsilon})}{|y-x_{\epsilon}|^{\alpha}} =0.
$$
But this implies that $u_{\epsilon}(x_{\epsilon}) \leq u_{\epsilon}(y)$ for all $y \in \Rn \setminus \Omega_{\epsilon}$ and from Lemma \ref{lem:strongmaxprinc} it follows that $u_{\epsilon}$ is a constant function for all $\epsilon < \epsilon_0$.We move on to the second case and therefore assume that there exists a sequence $x_{\epsilon}$ such that the corresponding sequence $y_{\epsilon} \to \infty$ as $\epsilon \to 0$. According to Lemma \ref{lem:A4} this means that the function $u$ is constant which implies that the function $u_{\epsilon}$ is also constant. Then
\begin{equation*}
\inf_{y \in \Rn \setminus \Omega_{\epsilon}} \frac{\varphi_{\epsilon}(y) - \varphi_{\epsilon}(x_{\epsilon})-M_{\epsilon}/4}{|y-x_{\epsilon}|^{\alpha}} = \inf_{y \in \Rn \setminus \Omega_{\epsilon}} \frac{-M_{\epsilon}/4}{|y-x_{\epsilon}|^{\alpha}} = -\frac{M_{\epsilon}/4}{\inf_{y \in \Rn \setminus \Omega_{\epsilon}} |y-x_{\epsilon}|^{\alpha}}.
\end{equation*}
Therefore, in this case \eqref{eq:wantthis} holds with
$$
\delta_{\epsilon} = \frac{M_{\epsilon}/4}{\inf_{y \in \Rn \setminus \Omega_{\epsilon}} |y-x_{\epsilon}|^{\alpha}} \to  \frac{M/4}{\inf_{y \in \Rn \setminus \Omega} |y-x_0|^{\alpha}} >0,
$$
for all $\epsilon$ less than some $\epsilon_0$.

In conclusion, from \eqref{eq:Lplus} and \eqref{eq:wantthis} it follows that for $\epsilon$ small enough, we can find a $\delta > 0$ such that
\begin{align*}
    \Linf \tilde{\varphi}_{\epsilon}(x_{\epsilon}) &\leq \Linf \varphi_{\epsilon}(x_{\epsilon})- \delta  =\Linf u_{\epsilon}(x_{\epsilon}) - \delta  \label{eq:rhs_cont}
     \leq  \tilde{f}^{\epsilon}(x_{\epsilon})-\delta  \\
    & =f(x_{\epsilon}^*) -\delta  \leq \sup_{y \in B_{r(\epsilon)}(x_{\epsilon})} f(y)-\delta \leq f(x_{\epsilon})-\tilde{\delta}, \nonumber
\end{align*}
for some $\delta \geq \tilde{\delta}>0$. Note that we have used the continuity of the function $f$ in the last inequality. Therefore,
\begin{equation}\label{eq:contradiction3}
   \Linf \tilde{\varphi}_{\epsilon}(x_{\epsilon}) < f(x_{\epsilon}), 
\end{equation}
and thus, $\tilde{\varphi}$ is a strict supersolution.

\textbf{Step 3. Conclusion and removing the boundedness assumption}\\
We note that \eqref{eq:contradiction3} contradicts \eqref{eq:contradiction} which implies that \eqref{eq:contradiction2} is false. This proves the theorem under the assumption that $u$ is bounded in $\Omega$. 

We are now going to remove that assumption. Since $u$ is bounded on $\Rn \setminus \Omega$, i.e., $|u(x)|\leq C$ for $x \in \Rn \setminus \Omega$, it follows from Lemma \ref{lem:strongmaxprinc} that $u$ is bounded from below in $\Omega$. We next choose a constant $K$ so large that $K \geq C$ and 
\begin{equation}\label{eq:infext}
  \inf_{y \in \Rn \setminus \Omega} \frac{u(y)-K}{|y-x_0|^{\alpha}} \leq f(x_0),  
\end{equation}
for all $x_0 \in \Omega$. This can be done since $u$ is bounded on $\Rn \setminus \Omega$ and $f(x)$ is bounded. Indeed, we consider any $y \in \Rn \setminus \Omega$ and let $d_y$ be the maximum distance from $y$ to $\Omega$, i.e., $d_y = \sup_{x \in \Omega} |y-x|$. Then we can for instance choose $K$ such that
$$
\frac{u(y)-K}{d^{\alpha}_y} \leq \inf_{x \in \Omega} f(x).
$$
We next define
\begin{align*}
   \tilde{u}(x) =   \begin{cases}
                    \textup{min}\{u, K\} & \textup{in  } \ \Omega, \\
                    u    & \textup{on  } \ \Rn \setminus \Omega.
     \end{cases}
\end{align*}
The function $\tilde{u}$ is lower semicontinuous since $u$ is, the minimum of $u$ and $K$ is and $K \geq u$ on $\partial \Omega$. Furthermore, due to the choice of $K$ in \eqref{eq:infext}, $\tilde{u}$ is a viscosity supersolution to \eqref{eq:Luf2} with $f$ on the right hand side. Since it is bounded we can use the proof above to see that $v \leq \tilde{u}$ and therefore $v \leq u$. This concludes the proof of the theorem.
\end{proof}

Since a viscosity solution to \eqref{eq:Luf} is both a viscosity supersolution and a viscosity subsolution, the uniqueness follows immediately. We formulate this as a corollary.

\begin{cor}
    Let $u$ be a viscosity solution to \eqref{eq:Luf} where $0 < \alpha \leq 1$, and $f \leq 0$ is bounded and continuous in $\Omega$. Furthermore, let $u(x)$ be bounded on $ \Rn \setminus \Omega$. In addition, we assume that $u(x)$ has a limit, as $x$ tends to infinity, in the sense that $\lim_{|x| \to \infty} u(x) = C_1$. Then $u$ is unique.
\end{cor}

\appendix
\section{Appendix}
In this section, we are going to prove a few lemmas that will be used in the proof of the comparison principle and we therefore assume the same conditions and notation as in Theorem \ref{thm:comparison}. Thus, we recall that
\begin{align*}
    M = \sup_{x \in \Rn} (v(x)-u(x)),
\end{align*}
and
\begin{align*}
    K_0 = \{x \in \Rn:v(x)-u(x) = M \},
\end{align*}
where $ M \geq \gamma > 0$.
\begin{lem}
\label{lem:A1}
    Let $u(x)$ be bounded on $ \Rn$.
Then $M < \infty$ and $K_0 \subset \subset \Omega$.
\end{lem}

\begin{proof}
 Since $v$ is a subsolution, $-v$ is a supersolution, and it follows from Lemma \ref{lem:strongmaxprinc} that the minimum of $-v$ is obtained outside of $\Omega$ or the function is constant. If $-v$ is constant, then $v$ is obviously bounded, and if it is not constant it follows that the maximum of $v$ is obtained outside of $\Omega$. Since $v \leq u \leq L$ on $\Rn \setminus \Omega$ it follows that $v(x) \leq L$ in $\Omega$. Thus, $v$ is bounded from above and since $u$ is bounded by assumption we clearly have that $M < \infty$.

Since $u(x) \geq v(x)$ on $\Rn \setminus \Omega$, clearly $K_0 \subset \Omega$ and since $u$ is lower semicontinuous and $v$ is upper semicontinuous it follows that $K_0 \subset \subset \Omega$. Indeed, if $K_0$ is not compactly contained in $\Omega$, then there exists $x_0 \in \partial \Omega$ such that $x_0 \in \ol{K_0}$. For any $\epsilon > 0$ we have the following: by the upper semicontinuity of $v$, there is a neighbourhood $V_\epsilon = V_\epsilon(x_0)$ of $x_0$ in $\Omega$ such that $v(y) \leq v(x_0) + \epsilon$ for $y \in V_\epsilon$ and by the lower semicontinuity of $u$, there exists a neighbourhood $U_\epsilon = U_\epsilon(x_0)$ of $x_0$ in $\Omega$, such that $u(y) \geq u(x_0) - \epsilon$ for $y \in U_\epsilon$. Then, $V_\epsilon \cap U_\epsilon$ is a neighbourhood of $x_0$ in which $v(y) - u(y) \leq v(x_0) - u(x_0) + 2\epsilon \leq 2\epsilon$, since $u \geq v$ on $\R^n \setminus \Omega$. By picking $\epsilon < \gamma /2$, we get a contradiction and hence that $K_0 \subset \subset \Omega$. 
\end{proof}
For $x \in \Rn$ we define the infimal convolution of $u$ as in Definition \ref{def:infconv} and we recall that $u_{\epsilon}$ is a supersolution to $\Linf u_{\epsilon} =  \tilde{f}^{\epsilon}$ in 
$$
\Omega_{\epsilon} = \{x \in \Omega:\dist(x, \partial \Omega) > r(\epsilon) \},
$$
where $r(\epsilon) = \sqrt{2L\epsilon}$, and 
$$
 \tilde{f}^{\epsilon}(x) =  f(x^*) \leq \sup_{y \in B_{r(\epsilon)}(x)} f(y).
$$
Here $x^*$ is the point where the infimum is attained for $u_{\epsilon}(x)$. Furthermore, we recall that
\begin{align*}
    M_{\epsilon} &= \sup_{x \in \Rn}(v(x)-u_{\epsilon}(x)) \geq M, \\
    K_{\epsilon} &= \{x \in \Rn  :v(x)-u_{\epsilon}(x) = M_{\epsilon} \}.
\end{align*}
\begin{lem}\label{lem:A2}
      For some small enough $\beta>0$ we consider a compact $\beta$-neighbourhood $K(\beta)$ of $K_0$. Then $K_{\epsilon} \subset \subset K(\beta)$ for sufficiently small values of $\epsilon$.
\end{lem}

\begin{proof}
     In order to do so, we are going to show that every sequence of points in $K_{\epsilon}$ will converge to a point in $K_0$. We first show that there exists some $\epsilon_0$ such that the set $K_{\epsilon}$ is bounded for  $\epsilon \leq \epsilon_0$. This can be seen as follows: If it is not true, there exists a sequence ${\epsilon}_k \to 0$ as $k \to \infty$ such that $x_{{\epsilon}_k} \in K_{{\epsilon}_k}$ and $|x_{\epsilon_k}| \to \infty$ as $k \to \infty$, and
\begin{align*}
   M_{\epsilon_k} &=  v(x_{\epsilon_k})-u_{\epsilon_k}(x_{\epsilon_k}).
\end{align*}
Since $|x_{\epsilon_k}^*| \to \infty $ as $|x_{\epsilon_k}|\to \infty$ we see that
$$
\liminf_{k \to \infty} u_{\epsilon_k}(x_{\epsilon_k}) \geq \liminf_{k \to \infty} u(x_{\epsilon_k}^*)\geq C_1.
$$
 Therefore, it follows that
\begin{align*}
    M\leq \limsup_{k \to \infty} \left(v(x_{\epsilon_k})-u_{\epsilon_k}(x_{\epsilon_k})\right) \leq \limsup_{k \to \infty} v(x_{\epsilon_k}) - \liminf_{k\to \infty} u_{\epsilon_k}(x_{\epsilon_k}) \leq C_2-C_1 \leq 0,
\end{align*}
where we have used that $\liminf_{|x| \to \infty} u(x) \geq C_1$ and $\limsup_{|x| \to \infty} v(x) \leq C_2$ where $C_1 \geq C_2$. This is obviously a contradiction since we have assumed that $M >0$. Thus, any sequence of points $x_{\epsilon} \in K_{\epsilon}$ has a subsequence that converges to some $x_0$. We therefore see that
 $$
 v(x_0) \geq \limsup_{x_{\epsilon} \to x_0} v(x_{\epsilon}) \geq \liminf_{x_{\epsilon} \to x_0} u(x_{\epsilon}) + M_{\epsilon} \geq u(x_0) +M,
 $$
where we have used that $v$ is upper semincontinuous in the first inequality and that $u$ is lower semicontinuous in the last inequality. Thus, $v(x_0)-u(x_0) \geq M$ and from the definition of $M$ it follows that $v(x_0)-u(x_0) = M$ so we conclude that $x_0 \in K_0$. Consequently, since $K_0 \subset \subset K(\beta)$ it follows that $K_{\epsilon} \subset \subset K(\beta)$ for $\epsilon$ small enough.
\end{proof}

\begin{lem}\label{lem:A3}
    There exist an $\epsilon_0>0$ and a $\tau_0 >0$ such that for $\epsilon < \epsilon_0$, and $\tau \leq \tau_0$ we have that
\begin{equation}
\label{eq:blahblah2}
v \leq u_{\epsilon} + M_{\epsilon}/8 \ \textup{ on } \ \Rn \setminus \Omega_{\tau},   
\end{equation}
where
$$
\Omega_{\tau}:= \{ x \in \Omega, d(x, \partial \Omega)> \tau \}. 
$$
\end{lem}
\begin{proof}
    If \eqref{eq:blahblah2} does not hold, then there exists a sequence $\epsilon_k \to 0$ as $k \to \infty$, such that for every $\epsilon_k$ there exists a sequence $\tau_i:=\tau_i(\epsilon_k) \to 0 $ as $i \to \infty$, and a sequence of points $x_i(\epsilon_k) \in \Rn \setminus \Omega_{\tau_i}$ such that \eqref{eq:blahblah} does not hold i.e., 
$$
v(x_i(\epsilon_k)) > u_{\epsilon_k}(x_i(\epsilon_k)) + M_{\epsilon_k}/8.
$$
We form the diagonal sequence $x_j:=x_j(\epsilon_j)$. If the sequence $x_j$ is such that $|x_j| \leq C$, then there exists a subsequence that we will continue to denote by $x_j$ such that $x_j \to \tilde{x} \in \Rn \setminus \Omega$. We then see that
\begin{align*}
    v(\tilde{x}) & \geq \limsup_{j\to \infty} v(x_j) > \liminf_{j \to \infty} (u_{\epsilon_j}(x_j) + M_{\epsilon_j}/8) \geq \liminf_{j \to \infty} (u(x_j^*) + M_{\epsilon_j}/8) \\
    &\geq u(\tilde{x}) + M/8,
\end{align*}
where we have used that $v$ is upper semicontinuous in the first inequality and that $u$ is lower semicontinuous in the last inequality. Thus, $v(\tilde{x}) > u(\tilde{x})$ which is clearly a contradiction since $u(x) \geq v(x)$ on $\Rn \setminus \Omega$. We next assume that the sequence of points $|x_j| \to \infty$ as $j \to \infty$ such that
\begin{align*}
C_2 \geq \limsup_{j\to \infty} v(x_j) & > \liminf_{j \to \infty} u_{\epsilon_j}(x_j) + M_{\epsilon_j}/8 \geq \liminf_{j \to \infty} u(x_j^*) + M_{\epsilon_j}/8 \geq C_1 +M/8.
\end{align*}
This is clearly a contradiction since $C_2 \leq C_1$ and $M>0$. This proves equation \eqref{eq:blahblah2}. 
\end{proof}

\begin{lem}\label{lem:A4}
    Let $u(x)$ be bounded on $ \Rn$ and let the sequence $x_{\epsilon} $ be such that $x_{\epsilon} \rightarrow x_0$ for some $x_0 \in K_0$. Furthermore, assume that $\Lneg u_{\epsilon}(x_{\epsilon})$ is attained for some $y_{\epsilon} \in \Rn \setminus \Omega_{\epsilon}$ and that $y_{\epsilon} \to \infty$ as $\epsilon \to 0$. Then $u$ is a constant function.
\end{lem}

\begin{proof}
     Since $u$ is bounded, it follows that 
\begin{equation}
\label{eq:liminfu}
\lim_{\epsilon \to 0} \inf_{y \in \Rn \setminus \Omega_{\epsilon}} \frac{\varphi_{\epsilon}(y)-\varphi_{\epsilon}(x_{\epsilon})}{|y-x_{\epsilon}|^{\alpha}} =0.
\end{equation}
We are going to show that this implies that
\begin{equation}
\label{eq:infu_limit}
    \inf_{y \in \Rn \setminus \Omega} \frac{u(y)-u(x_0)}{|y-x_0|^{\alpha}} =0.
\end{equation}
 We therefore assume that \eqref{eq:infu_limit} does not hold which implies that the left hand side is strictly less than zero, i.e., there exist some $y_0 \in \Rn \setminus \Omega$ and some $a>0$ such that
\begin{equation*}
    \frac{u(y_0)-u(x_0)}{|y_0-x_0|^{\alpha}} = -a < 0.
\end{equation*}
Due to the upper semicontinuity of $v$ it follows that
\begin{equation*}
\limsup_{\epsilon \to 0} u_{\epsilon}(x_{\epsilon}) =\limsup_{\epsilon \to 0} (v(x_{\epsilon}) - M_{\epsilon}) \leq v(x_0) -M = u(x_0),
\end{equation*}
 and due to the lower semicontinuity of $u$, and the fact that $x_{\epsilon}^* \to x_0$ as $\epsilon \to 0$, since $|x_{\epsilon}^*-x_{\epsilon}| \to 0$, we obtain
\begin{equation*}
 \liminf_{\epsilon \to 0} u_{\epsilon}(x_{\epsilon}) =\liminf_{\epsilon \to 0} \left(u(x^*_{\epsilon}) + \frac{|x^*_{\epsilon}-x_{\epsilon}|^2}{2 \epsilon} \right)   \geq \liminf_{\epsilon \to 0} u(x_{\epsilon}^*) \geq u(x_0).
\end{equation*}
In conclusion, 
$$
\lim_{\epsilon \to 0} u_{\epsilon} (x_{ \epsilon}) = u(x_0).
$$
Thus, for a fixed $y \in \Rn \setminus \Omega$ we see that
\begin{align*}
\label{eq:pointwiseconv}
    \frac{u_{\epsilon}(y)-u_{\epsilon}(x_{\epsilon})}{|y-x_{\epsilon}|^{\alpha}} &= \frac{u(y)-u(x_0)}{|y-x_0|^{\alpha}} \cdot \frac{|y-x_0|^{\alpha}}{|y-x_{\epsilon}|^{\alpha}} + \frac{u_{\epsilon}(y)-u(y) + u(x_0)-u_{\epsilon}(x_{\epsilon})}{|y-x_{\epsilon}|^{\alpha}} \nonumber\\
    &\to \frac{u(y)-u(x_0)}{|y-x_0|^{\alpha}}, 
\end{align*}
as $\epsilon \to 0$. It therefore follows that there exists $\epsilon_0 >0$ such that for every $\epsilon \leq \epsilon_0$ we have that
\begin{equation*}
   \left| \frac{u_{\epsilon}(y_0)-u_{\epsilon}(x_{\epsilon})}{|y_0-x_{\epsilon}|^{\alpha}}-\frac{u(y_0)-u(x_0)}{|y_0-x_0|^{\alpha}} \right| < \frac{a}{2},
\end{equation*}
which implies that 
\begin{equation*}
    \frac{u_{\epsilon}(y_0)-u_{\epsilon}(x_{\epsilon})}{|y_0-x_{\epsilon}|^{\alpha}} \leq -\frac{a}{2} <0,
\end{equation*}
for $\epsilon \leq \epsilon_0$. But under the assumption that $y_{\epsilon} \to \infty$ this contradicts \eqref{eq:liminfu} and therefore \eqref{eq:infu_limit} must hold.
 But this implies that $u(y) \geq u(x_0)$ for all $y \in \Rn \setminus \Omega$ and from Lemma \ref{lem:strongmaxprinc} we deduce that $u$ must be a constant function.
\end{proof}

\bibliographystyle{abbrv}
\bibliography{bibliography}

\begin{thebibliography}{10}

\bibitem{AS2010}
S.~N. Armstrong and C.~K. Smart.
\newblock An easy proof of {J}ensen's theorem on the uniqueness of infinity
  harmonic functions.
\newblock {\em Calc. Var. Partial Differential Equations}, 37(3-4):381--384,
  2010.

\bibitem{A1967}
G.~Aronsson.
\newblock Extension of functions satisfying {L}ipschitz conditions.
\newblock {\em Ark. Mat.}, 6:551--561 (1967), 1967.

\bibitem{BB2001}
G.~Barles and J.~Busca.
\newblock Existence and comparison results for fully nonlinear degenerate
  elliptic equations without zeroth-order term.
\newblock {\em Comm. Partial Differential Equations}, 26(11-12):2323--2337,
  2001.

\bibitem{BDM1991}
T.~Bhattacharya, E.~DiBenedetto, and J.~Manfredi.
\newblock Limits as {$p\to\infty$} of {$\Delta_pu_p=f$} and related extremal
  problems.
\newblock Number Special Issue, pages 15--68 (1991). 1989.
\newblock Some topics in nonlinear PDEs (Turin, 1989).

\bibitem{BCF2012}
C.~Bjorland, L.~Caffarelli, and A.~Figalli.
\newblock Nonlocal tug-of-war and the infinity fractional {L}aplacian.
\newblock {\em Comm. Pure Appl. Math.}, 65(3):337--380, 2012.

\bibitem{CRY2010}
L.~Caffarelli, J.-M. Roquejoffre, and Y.~Sire.
\newblock Variational problems with free boundaries for the fractional
  laplacian.
\newblock {\em Journal of the European Mathematical Society},
  012(5):1151--1179, 2010.

\bibitem{CS2007}
L.~Caffarelli and L.~Silvestre.
\newblock An extension problem related to the fractional {L}aplacian.
\newblock {\em Comm. Partial Differential Equations}, 32(7-9):1245--1260, 2007.

\bibitem{CLM2012}
A.~Chambolle, E.~Lindgren, and R.~Monneau.
\newblock A {H}\"{o}lder infinity {L}aplacian.
\newblock {\em ESAIM Control Optim. Calc. Var.}, 18(3):799--835, 2012.

\bibitem{CP2007}
F.~Charro and I.~Peral.
\newblock Limit branch of solutions as {$p\to\infty$} for a family of
  sub-diffusive problems related to the {$p$}-{L}aplacian.
\newblock {\em Comm. Partial Differential Equations}, 32(10-12):1965--1981,
  2007.

\bibitem{CEG2001}
M.~G. Crandall, L.~C. Evans, and R.~F. Gariepy.
\newblock Optimal {L}ipschitz extensions and the infinity {L}aplacian.
\newblock {\em Calc. Var. Partial Differential Equations}, 13(2):123--139,
  2001.

\bibitem{CGW2007}
M.~G. Crandall, G.~Gunnarsson, and P.~Wang.
\newblock Uniqueness of {$\infty$}-harmonic functions and the eikonal equation.
\newblock {\em Comm. Partial Differential Equations}, 32(10-12):1587--1615,
  2007.

\bibitem{SR2019}
J.~a.~V. da~Silva and J.~D. Rossi.
\newblock The limit as {$p\to\infty$} in free boundary problems with fractional
  {$p$}-{L}aplacians.
\newblock {\em Trans. Amer. Math. Soc.}, 371(4):2739--2769, 2019.

\bibitem{BPLS2022}
J.~Fern\'{a}ndez~Bonder, M.~P\'{e}rez-Llanos, and A.~M. Salort.
\newblock A {H}\"{o}lder infinity {L}aplacian obtained as limit of {O}rlicz
  fractional {L}aplacians.
\newblock {\em Rev. Mat. Complut.}, 35(2):447--483, 2022.

\bibitem{FL2016}
R.~Ferreira and M.~P\'{e}rez-Llanos.
\newblock Limit problems for a fractional {$p$}-{L}aplacian as {$p\to\infty$}.
\newblock {\em NoDEA Nonlinear Differential Equations Appl.}, 23(2):Art. 14,
  28, 2016.

\bibitem{IL2005}
H.~Ishii and P.~Loreti.
\newblock Limits of solutions of {$p$}-{L}aplace equations as {$p$} goes to
  infinity and related variational problems.
\newblock {\em SIAM J. Math. Anal.}, 37(2):411--437, 2005.

\bibitem{Jensen1993}
R.~Jensen.
\newblock Uniqueness of {L}ipschitz extensions: minimizing the sup norm of the
  gradient.
\newblock {\em Arch. Rational Mech. Anal.}, 123(1):51--74, 1993.

\bibitem{JP2012}
V.~Julin and P.~Juutinen.
\newblock A new proof for the equivalence of weak and viscosity solutions for
  the {$p$}-{L}aplace equation.
\newblock {\em Comm. Partial Differential Equations}, 37(5):934--946, 2012.

\bibitem{J1998}
P.~Juutinen.
\newblock Minimization problems for {L}ipschitz functions via viscosity
  solutions.
\newblock {\em Ann. Acad. Sci. Fenn. Math. Diss.}, (115):53, 1998.
\newblock Dissertation, University of Jyv\"{a}skul\"{a}, Jyv\"{a}skul\"{a},
  1998.

\bibitem{JPR2016}
P.~Juutinen, M.~Parviainen, and J.~D. Rossi.
\newblock Discontinuous gradient constraints and the infinity {L}aplacian.
\newblock {\em Int. Math. Res. Not. IMRN}, (8):2451--2492, 2016.

\bibitem{LL2012}
E.~Lindgren and P.~Lindqvist.
\newblock Fractional eigenvalues.
\newblock {\em Calc. Var. Partial Differential Equations}, 49(1-2):795--826,
  2014.

\bibitem{LM1995}
P.~Lindqvist and J.~J. Manfredi.
\newblock The {H}arnack inequality for {$\infty$}-harmonic functions.
\newblock {\em Electron. J. Differential Equations}, pages No. 04, approx. 5
  pp.\, 1995.

\bibitem{LM2006}
P.~Lindqvist and J.~J. Manfredi.
\newblock Viscosity supersolutions of the evolutionary {$p$}-{L}aplace
  equation.
\newblock {\em Differential Integral Equations}, 20(11):1303--1319, 2007.

\bibitem{LW2008}
G.~Lu and P.~Wang.
\newblock Inhomogeneous infinity {L}aplace equation.
\newblock {\em Adv. Math.}, 217(4):1838--1868, 2008.

\bibitem{STD2018}
A.~Schikorra, T.-T. Shieh, and D.~E. Spector.
\newblock Regularity for a fractional {$p$}-{L}aplace equation.
\newblock {\em Commun. Contemp. Math.}, 20(1):1750003, 6, 2018.

\bibitem{Y2010}
Y.~Yu.
\newblock Maximal and minimal solutions of an {A}ronsson equation: {$L^\infty$}
  variational problems versus the game theory.
\newblock {\em Calc. Var. Partial Differential Equations}, 37(1-2):63--74,
  2010.

\end{thebibliography}

\end{document}